\documentclass{article}

\usepackage{amsmath}
\usepackage{amssymb}
\usepackage{amsthm, amscd}
\usepackage{amsmath,amsthm,amsfonts,amssymb}

\usepackage[usenames]{color}

\newtheorem{thm}{Theorem}[section]
\newtheorem{proposition}[thm]{Proposition}
\newtheorem{lemma}[thm]{Lemma}
\newtheorem{corollary}[thm]{Corollary}

\theoremstyle{remark}
\newtheorem{rem}[thm]{\bf Remark}

\newtheorem{prob}{Problem}

\newcommand{\mbb}{\mathbb}
\newcommand{\fra}{\mathfrak}

\newcommand{\ds}{\displaystyle}

\newcommand{\N}{\mathbb{N}}
\newcommand{\Z}{\mbb{Z}}
\newcommand{\ZZ}{\{\pm I_n\}}
\newcommand{\Q}{\mbb{Q}}
\newcommand{\D}{\mathcal{D}}

\newcommand{\U}{\fra{U}}
\newcommand{\B}{\mathfrak{B}}
\newcommand{\CP}{\mathcal{P}}

\newcommand{\QM}{\mathbb{Q}^\times}
\newcommand{\KM}{K^\times}

\newcommand{\OR}{\mathcal{O}}
\newcommand{\OM}{\mathcal{O}^{\times}}
\newcommand{\RM}{R^{\times}}
\newcommand{\SM}{S^{\times}}
\newcommand{\define}{\stackrel{\text{def}}{=}}

\usepackage{fancyhdr}

\begin{document}
\title{On  groups elementarily equivalent to a group of triangular matrices $T_n(R)$}

\author{ Alexei Miasnikov,
 Mahmood Sohrabi\footnote{Address: Stevens Institute of Technology, Department of Mathematical Sciences, Hoboken, NJ 07087, USA. Email: msohrab1@stevens.edu }}

\maketitle
\begin{abstract}
In this paper we investigate the structure of groups elementarily equivalent to the group $T_n(R)$ of all  invertible upper triangular $n\times n$ matrices, where $n\geq 3$ and $R$ is a characteristic zero integral domain. In particular we give both necessary and sufficient conditions for a group being elementarily equivalent to $T_n(R)$ where $R$ is a characteristic zero algebraically closed field, a real closed field, a number field, or the ring of integers of a number field.  
\end{abstract}

 \section{Motivation}
 Given an algebraic structure $\mathfrak{U}$ one can ask if the first-order theory of $\mathfrak{U}$  is decidable, or what are the structures (perhaps under some restrictions) which have the same first-order  theory as $\mathfrak{U}$. A. Tarski posed several problems of this nature in the 1950's. These type of problems are referred to as \emph{Tarski-type problems} or simply  \emph{Tarski problems}.

Tarski-type  problems on groups, rings, and other algebraic structures have been very inspirational  and have led to some  important developments in modern algebra and model theory. 
Usually solutions to  these   problems for a   structure $\mathfrak{U}$ clarify the most fundamental algebraic properties of $\mathfrak{U}$.   Indeed, it suffices to mention here results on first-order theories   of algebraically closed fields,  real closed fields \cite{Tarski2}, the fields of $p$-adic  numbers \cite{Ax-Kochen, Ershov}, abelian groups and modules \cite{Szmielewa, Baur}, boolean algebras \cite{Tarski3, Ershov3}, and free and hyperbolic groups \cite{KM1, KM2, Sela, Sela8}. 
 We refer the reader to~\cite{MS2010,ICM} for a brief survey of the history of Tarski-type problems in groups. 

We propose three  specific Tarski-type problems here:

\begin{prob}
Given a classical linear group $G_m(K)$ over a field $K$, where $G \in \{GL,SL,PGL,PSL, \}$ and $m \geq 2$, characterize all groups elementarily equivalent to $G_m(K)$ .
\end{prob}
\begin{prob} Given a (connected) solvable linear algebraic group $G$ characterize all groups elementarily equivalent to $G$.\end{prob}

\begin{prob} Given an arbitrary polycyclic-by-finite group $G$ characterize all groups elementarily equivalent to $G$.\end{prob}

A restricted version of Problem 1 was introduced and initially studied by Malcev, who proved that in,   the notation above,  $G_m(K_1) \equiv G_n(K_2)$ (where $K_1, K_2$ are fields of characteristic zero) if and only if  $m = n$ and $K_1 \equiv K_2$. In a series of papers Bunina and Mikhalev extended Malcev's results for other rings and groups (see \cite{BM}). In general,  not much is known  about Problem 1. The main difficulty here is that given a group, say $H$,  which is first-order equivalent to  $G_m(K)$  one has to attempt to recover the ``linear"  structure of $H$ only by first-order formulas,  and  only then  may apply Malcev's theorem. 

Problems 2 and 3 are also wide open, even though there are some relevant results. C. Lasserre and F. Oger~\cite{LO} give a criterion for elementary equivalence of two polycyclic groups.  Problem 3 is still open even in the case where $G$ is finitely generated nilpotent.  In \cite{beleg99} O. Belegradek described groups elementarily equivalent to a given nilpotent group $UT_n(\mathbb{Z})$,  and the authors of the present paper described all groups elementarily equivalent to a free nilpotent group of finite rank  \cite{MS2009, MS2010}.  In  \cite{MS2012} we  developed  techniques which seem to be useful in tackling Problem 3 in the nilpotent case. 
O. Fr\'econ~\cite{FR} considers the problem of elementary equivalence and description of abstract isomorphisms of algebraic groups, but the ground fields are always assumed to be algebraically closed, which allows one to use the technique of finite Morley rank and alike.

 This paper contributes to the study of the above problems in the following ways. Firstly, we present a framework to approach these and similar problems via  nilpotent radicals in solvable groups. Secondly, we solve these problems for  the group of all invertible $n\times n$ upper triangular matrices $T_n(R)$ over a ring $R$ (under some restrictions on $R$), which is interesting in its own right, but also  demonstrating that the approach works. Even though these results may look particular, we believe otherwise, since the groups $T_n(R)$ are typical, ``model" representatives of the groups in the problems. Besides,  the groups $T_n(R)$, as they are, play an important  part  in the study of model theory of groups $G_m(K)$ from Problem 1, in fact, this study to some extent directly depends on the understanding of model theory of groups $T_n(R)$.  

In Section~\ref{prelim:sec} we shall give a quick review of the basic notation and concepts needed to understand the main results of this paper. In Section~\ref{results:sec} we describe these main results and the structure of the rest of the paper.

\section{Preliminaries and statements of the main results}
 \label{prelim:sec}

\subsection{Preliminaries}

\subsubsection{Basic group-theoretic and ring-theoretic notation} 
For a group $G$ by $Z(G)={x\in G:xy=yx, \forall y\in G}$ we mean the center of $G$. the derived subgroup $G'$ of $G$ is the subgroup of $G$ generated by all commutators $[x,y]=x^{-1}y^{-1}xy$ of elements $x$ and $y$ of $G$. We also occasionally use $x^y$ for $y^{-1}xy$, for $x$ and $y$ in $G$.

All rings in this paper are commutative associative with unit. We denote the ring of integers by $\Z$ and the field of rationals by $\Q$. By a \emph{number field} we mean a finite extension of $\Q$. By \emph{the ring of integers $\OR$ of a number field $F$} we mean the subring of $F$ consisting of all roots of monic polynomials with integer coefficients. For a ring $R$, by $\RM$ we mean the multiplicative group of invertible (unit) elements of $R$. By $R^+$ we mean the additive group of $R$. 
    
\subsubsection{ Extensions and 2-cocycles} Assume that $A$ is an abelian group and  $B$ is a group. A function \[f:B\times B\rightarrow A\] satisfying
\begin{itemize}
 \item $f(xy,z)f(x,y)=f(x,yz)f(y,z)$, \quad $\forall x,y,z\in B,$
 \item $f(1,x)=f(x,1)=1$, $\forall x\in B$,
\end{itemize}
is called a \textit{2-cocycle}. If $B$ is abelian a 2-cocycle $f:B\times B\rightarrow A$ is \textit{symmetric} if it also satisfies the identity:
\[f(x,y)=f(y,x)\quad \forall x,y \in B.\]
By an extension of $A$ by $B$ we mean a short exact sequence of groups
\[1\rightarrow A \xrightarrow{\mu} E \xrightarrow{\nu} B \rightarrow 1,\]
where $\mu$ is the inclusion map. The extension is called \emph{abelian} if $E$ is abelian and it is called \emph{central} if $A\leq Z(E)$. A \textit{2-coboundary}\index{ $2$-coboundary} $g:B\times B \rightarrow A$ is a 2-cocycle satisfying :
 \[\psi(xy)=g(x,y)\psi(x)\psi(y), \quad \forall x,y\in B,\]
 for some function $\psi:B\rightarrow A$. One can make the set $Z^2(B,A)$ of all 2-cocycles  and the set $B^2(B,A)$ of all 2-coboundaries into
 abelian groups in an obvious way. Clearly $B^2(B,A)$
 is a subgroup of $Z^2(B,A)$. Let us set \[H^2(B,A)= Z^2(B,A)/B^2(B,A).\]
Assume $f$ is a 2-cocycle. Define a group $E(f)$ by $E(f)=B\times A$ as sets with the multiplication
\[(b_1, a_1)(b_2, a_2)=(b_1b_2, a_1a_2f(b_1,b_2)) \quad \forall a_1,a_2\in A,\forall b_1,b_2 \in B.\]
The above operation is a group operation and the resulting extension is central. It is a well known fact that there is a bijection between the equivalence classes of central
extensions of $A$ by $B$ and elements of the group $H^2(B,A)$ given by assigning $f \cdot B^2(B,A)$ the equivalence class of $E(f)$.

If $B$ is abelian $f\in Z^2(B,A)$  is symmetric if and only if it arises from an abelian extension of $A$ by $B$. As it
can be easily imagined there is a one to one correspondence between the equivalent classes of abelian extensions
and the quotient group \[Ext(B,A)=S^2(B,A)/(S^2(B,A)\cap B^2(B,A))\index{ $Ext(B,A)$},\] where $S^2(B,A)$\index{ $S^2(B,A)$} denotes the group of symmetric
2-cocycles. For further details we refer the reader to ~(\cite{robin}, Chapter 11). 

\subsubsection{Model-theoretic notation and terminology} 
Our reference for basic model theory is~\cite{hodges}.

A group $G$ is considered to be
the structure $\langle|G|,\cdot,^{-1},1\rangle$ where $\cdot$, $^{-1}$ and $1$, name multiplication, inverse operation
and the trivial element of the group respectively. We call the corresponding first-order language $L$.

Let $\mathfrak{U}$ be a structure and $\phi(x_1,\ldots, x_n)$ be a first-order formula of the signature of
$\mathfrak{U}$ with $x_1$,\ldots ,$x_n$ free variables. Let $(a_1,\ldots, a_n)\in |\mathfrak{U}|^n$. We denote
such a tuple by $\bar{a}$. The notation $\mathfrak{U} \models \phi(\bar{a})$ is intended to mean that the tuple
$\bar{a}$ satisfies $\phi(\bar{x})$ when $\bar{x}$ is an abbreviation for the tuple $(x_1,\ldots, x_n)$ of
variables. 

Given a structure $\mathfrak{U}$ and a first-order formula $\phi(x_1,\ldots, x_n)$ of the signature of
$\mathfrak{U}$, $\phi(\mathfrak{U}^n)$ refers to $\{\bar{a}\in |\mathfrak{U}|^n: \mathfrak{U}\models
\phi(\bar{a})\}$. Such a relation or set is called \textit{first-order definable without parameters} or \emph{absolutely definable}. If
$\psi(x_1,\ldots,x_n,y_1,\ldots,y_m)$ is a first-order formula of the signature of $\mathfrak{U}$ and $\bar{b}$
an $m$-tuple of elements of $\mathfrak{U}$ then $\psi(\mathfrak{U}^n,\bar{b})$ means $\{\bar{a}\in
|\mathfrak{U}|^n :\mathfrak{U}\models \psi(\bar{a},\bar{b})\}$. A set or relation like this is said to be
\textit{first-order definable with parameters}.

Now let $T$ be a theory of signature $\Delta$. Suppose that $S: Mod (T) \rightarrow K$ is a functor defined on the class $Mod (T)$ of all models of the theory $T$ (a category with isomorphisms) into a certain category $K$ of structures of signature $\Sigma$. If there exists a system of first-order formulas $\Psi$ of signature $\Delta$, which absolutely interprets the system $S(\frak{B})$ in any model $\frak{B}$ of the theory $T$ we say that $S(\frak{B})$ is \textit{absolutely interpretable in $\frak{B}$ uniformly with respect to $T$}. For a definition of interpretability see~\cite{hodges}.

Let $\mathfrak{U}$ be a structure of signature $\Sigma$. The \textit{elementary theory} $Th(\mathfrak{U})$ of the structure
$\mathfrak{U}$ is the set: \[\{\phi:\mathfrak{U}\models \phi, \phi \textrm{ a first-order sentence of signature
}\Sigma\}.\]
We say two structures $\mathfrak{U}$ and $\mathfrak{B}$ of signature $\Sigma$ are \textit{elementarily
equivalent} and write $\U\equiv \B$ if $Th(\mathfrak{U})=Th(\mathfrak{B})$.

\subsection{Statements of the main results}\label{results:sec}
The group $G=T_n(R)$ is a semi-direct product \[T_n(R)=D_n(R) \ltimes_{\phi_{n,R}} UT_n(R),\] where $D_n(R)$ is the subgroup of all diagonal matrices in $T_n(R)$, $UT_n(R)$ denotes the subgroup of all upper unitriangular matrices (i.e. upper triangular with 1's on the diagonal), and the homomorphism  $\phi_{n,R}: D_n(R) \to Aut(UT_n(R))$ describes the action of $D_n(R)$ on $UT_n(R)$ by conjugation. The subgroup $UT_n(R)$ is the so-called \emph{unipotent radical of $G$}, i.e. the subgroup consisting of all unipotent matrices in $G$, so sometimes we denote it as $G_u$. The subgroup $D_n(R)$ is a direct product $(\RM)^n$ of $n$ copies of the multiplicative group of units $\RM$ of $R$. The center $Z(G)$ of $G$ consists of diagonal scalar matrices $Z(G)=\{\alpha I_n: \alpha\in \RM\}\cong \RM$, where $I_n$ is the identity matrix. Again it is standard knowledge that $Z(G)$ is a direct factor of $D_n(R)$, i.e. there is a subgroup $B\leq D_n$ such that $D_n=B \times Z(G)$. Now we define a new group just by deforming the multiplication on $D_n$. Let $E_n=E_n(R)$ be an arbitrary abelian extension of $Z(G)\cong \RM$ by $D_n/Z(G)\cong (\RM)^{n-1}$. As it is customary in extension theory we can assume $E_n=D_n=B \times Z(G)$ as sets, while the product on $E_n$ is defined as follows:
\[(x_1,y_1)\cdot (x_2,y_2)=(x_1x_2,y_1y_2f(x_1,x_2)),\]
for a symmetric 2-cocycle $f\in S^2(B,Z(G)).$
 Now define a map $\psi_{n,R}: E_n\to Aut(UT_n(R))$ by
 \[\psi_{n,R}((x,y))\define\phi_{n,R}((x,y)),~~ (x,y)\in B \times Z(G).\]
 The definition actually makes sense since $ker(\phi_{n,R})=Z(G)$ and it is easy to verify that it is indeed a homomorphism.
Now define a  new group structure $H$ on the base set of $G$ by
 \[H\define E_n \ltimes_{\psi_{n,R}} UT_n(R).\]  We call such a group $H$ an \emph{abelian deformation of $T_n(R)$}.
 
 Indeed any abelian extension $E_n$ of $\RM$ by $(\RM)^{n-1}$, due to the fact that $ Ext((\RM)^{n-1},\RM)\cong \prod_{i=1}^{n-1} Ext(\RM,\RM)$, is uniquely determined by some symmetric 2-cocycles $f_i\in S^2(\RM,\RM)$, $i=1, \ldots , n-1$ up to equivalence of extensions. So if the $f_i$ are the defining 2-cocycles for $E_n$ above we also denote the group $H$ obtained above by $T_n(R, f_1, \ldots , f_{n-1})$ or $T_n(R,\bar{f})$. Indeed a more telling notation is to use $D_n(R,\bar{f})$ for $E_n(R)$ and therefore
 \[T_n(R,\bar{f})=D_n(R,\bar{f})\ltimes _{\psi_{n,R}}UT_n(R).\]
 
  We will study these groups in more detail in Section~\ref{abdef:sec}. 
 
We are now ready to state the main results of this paper.

\begin{thm}\label{mainthm} Let $G=T_n(R)$ be the group of invertible $n\times n$ upper triangular matrices over a characteristic zero integral domain $R$. If $H$ is an arbitrary group elementarily equivalent to $G$ then $H\cong T_n(S,f_1, \ldots, f_{n-1})$ for some ring $S\equiv R$ and symmetric 2-cocycles $f_i\in S^2(\SM,\SM)$.\end{thm}

The above theorem gives a necessary condition for a group $H$ to be elementarily equivalent to $T_n(R)$. As for sufficient condition(s) we need to first define a specific type of 2-cocycles. Given a ring $R$ as in the statement of the theorem above a symmetric 2-cocycle $f:\RM\times \RM \to \RM$ is said to be \emph{coboundarious on torsion} or $\emph{CoT}$ if the restriction $g:T\times T \to \RM$, where $T=T(\RM)$, of $f$ to $T\times T$ is a 2-coboundary.

\begin{thm} \label{main(improved):thm} Assume $R$ is an integral domain of characteristic zero where the maximal torsion subgroup $T(\RM)$ of $\RM$ is finite. Then for a group $H$
\[T_n(R)\equiv H \Leftrightarrow H\cong T_n(S, \bar{f}),\] 
for some ring $S\equiv R$ and some CoT 2-cocycles $f_i\in S^2(\SM,\SM)$, $i=1, \ldots ,n-1$.  \end{thm}

Since any number field or any ring of integers of a number field satisfies the hypotheses of Theorem~\ref{main(improved):thm}, the following result is immediate.

\begin{corollary}\label{CoT:cor} Assume $R$ is a number field or the ring of integers of a number field. Then $H\equiv T_n(R)$ if and only if $H\cong T_n(S,\bar{f})$ for some ring $S\equiv R$, where each $f_i$ is  CoT.\end{corollary} 

In case that $R$ is a characteristic zero algebraically closed field or a real closed field the introduction of abelian deformations is not necessary.

\begin{thm}\label{charthm-alg-realclosed:thm} Assume $F$ is a characteristic zero algebraically closed field or a real closed field. Then
 \[ H\equiv T_n(F) \Leftrightarrow H\cong T_n(K),\] 
for some field  $K\equiv F$.\end{thm}
 As for the necessity of introducing abelian deformations we prove the following theorems. 

\begin{thm}\label{elemnotiso-Q:thm} There is a countable field $K$, $K\equiv \Q$ and there are some $f_i\in S^2(\KM,\KM)$ such that $T_n(\Q)\equiv T_n(K, \bar{f})$ but $T_n(K, \bar{f})\ncong T_n(K')$ for any field $K'$.\end{thm} 

\begin{thm}\label{charthm-ringofinter:thm} Assume $\OR$ is the ring of integers of an algebraic number field.
\begin{enumerate}
\item If $\OM$ is finite, then a group $H$ is elementarily equivalent to $T_n(\OR)$ if and only if $H\cong T_n(R)$ for some ring $R\equiv \OR$.
\item If $\OM$ is infinite, then there exists a countable ring $R\equiv \OR$ and some $f_i\in S^2(\RM,\RM)$ such that $T_n(\OR)\equiv T_n(R,\bar{f})$ but $T_n(R,\bar{f})\ncong T_n(S)$ for any ring $S$.
\end{enumerate} 
\end{thm}
   

The rest of the paper is structured as follows. In Section~\ref{Tn(R):sec} we briefly discuss the structure of $T_n(R)$. More specifically we describe such a group by generators and relations. Section~\ref{abdef:sec} discusses abelian deformations $T_n(R,\bar{f})$ of $T_n(R)$ in some detail. In Section~\ref{fitt:sec} we discuss first-order definability of nilpotent subgroups and the Fitting subgroup of a group where the Fitting subgroup is itself nilpotent. Section~\ref{mainthm:sec}, clearly titled as so, provides a proof of Theorem~\ref{mainthm}. Section~\ref{suff:sec} provides a proof of Theorem~\ref{main(improved):thm}. In Section~\ref{closed:sec} we present a proof Theorem~\ref{charthm-alg-realclosed:thm}. In Section~\ref{abiso:sec} we discuss the conditions under which a group $T_n(R)$ and some abelian deformation $T_n(S,\bar{f})$ are isomorphic as groups. Sections~\ref{Tn(Q):sec} and~\ref{Tn(O):sec} provide proofs of Theorem~\ref{elemnotiso-Q:thm} and Theorem~\ref{charthm-ringofinter:thm}, respectively.

\section{Generators and relations for $T_n(R)$}\label{Tn(R):sec}
 Again we recall that the group $G=T_n(R)$ is isomorphic to a semi-direct product $D_n(R) \ltimes UT_n(R)$, where $UT_n(R)$ denotes the group of all $n\times n$ upper triangular unipotent matrices and $D_n(R)$ is the group of $n\times n$ diagonal matrices over $R$. Obviously $D_n=D_n(R)\cong (\RM)^n$ where $\RM$ denotes the group multiplicative units of $R$. In order to describe generators and relations for $T_n(R)$ we need those of $UT_n(R)$, $D_n$ and a description of the action of $D_n$ on $UT_n(R)$. To begin, let $e_{ij}$, $i<j$, be the matrix with $ij$'th entry $1$ and every other entry $0$, and let $t_{ij}=I_n+e_{ij}$, where $I_n$ is the $n\times n$ identity matrix. Let also $t_{ij}(\alpha)=I_n+\alpha e_{ij}$, for $\alpha\in R$. 

The $t_{ij}(\alpha)$'s, $1\leq i<j\leq n$, $\alpha\in R$, called \emph{transvections}, generate  $UT_n(R)$ and satisfy the well-known (Steinberg) relations:
\begin{enumerate}
\item $t_{ij}(\alpha)t_{ij}(\beta)=t_{ij}(\alpha+\beta), \forall \alpha,\beta\in R.$
\item \begin{equation*}[t_{ij}(\alpha),t_{kl}(\beta)]=\left\{\begin{array}{ll}
 t_{il}(\alpha\beta)  & \text{if } j=k \\
 t_{kj}(-\alpha\beta)  & \text{if } i=l \\
 I_n &\text{if } i\neq l, j\neq k
  \end{array}\right.
  \end{equation*}
for all $\alpha,\beta\in R$.
\end{enumerate} 
 Indeed these generators and relations define $UT_n(R)$ up to isomorphism. 
 
Let $diag[\alpha_1, \ldots ,\alpha_n]$ be the $n\times n$ diagonal matrix with $ii$'th entry $\alpha_i\in \RM$. The group $D_n(R)$ consists precisely of these elements as the $\alpha_i$ range over $\RM$. Now consider the following diagonal matrices
\[d_i(\alpha)\stackrel{\text{def}}{=} diag[1,\ldots ,  \underbrace{\alpha}_{i'\text{th}}, \ldots, 1],\]
and let us set \[d_i\define d_i(-1).\] Clearly the $d_i(\alpha)$ generate $D_n(F)$ as $\alpha$ ranges over $\RM$. The only relations they satisfy are:
\begin{enumerate}
\item $d_{i}(\alpha)d_{i}(\beta)=d_{i}(\alpha\beta) , \forall \alpha,\beta \in \RM$
\item $[d_i(\alpha),d_j(\beta)]=I_n, \forall \alpha,\beta\in \RM$,
\end{enumerate}
which illustrate the isomorphism $D_n(R)\cong(\RM)^n$.
The elements $d_i(\alpha)$, $1\leq i \leq n$, $\alpha\in \RM$ and $t_{kl}(\beta)$, $1\leq k<l\leq n$, $\beta\in R$ generate $T_n(R)$. To have a presentation for $T_n(R)$ all we need now is to describe the action of the $d_i(\alpha)$ on the $t_{kl}(\beta)$ by conjugation. 

Some simple matrix calculations show that:
 \begin{equation}\label{conj-act:eq}
 d_k(\alpha^{-1})t_{ij}(\beta)d_k(\alpha)=\left\{\begin{array}{ll}
 t_{ij}(\beta)  & \text{if } k\neq i, k\neq j \\
 t_{ij}(\alpha^{-1}\beta) &\text{if } k=i\\
 t_{ij}(\alpha\beta)&\text{if } k=j
 \end{array}\right.
 \end{equation}
  
So indeed we hinted towards a proof of the following lemma:

\begin{lemma}\label{TN-pres:lem} 
The group $T_n(R)$ is generated by  
\[\{d_i(\alpha),t_{kl}(\beta): 1\leq i \leq n, 1\leq k<l\leq n, \alpha\in \RM , \beta\in R\},\]
with relations:

\begin{enumerate}
\item $t_{ij}(\alpha)t_{ij}(\beta)=t_{ij}(\alpha+\beta), \forall \alpha,\beta\in R.$
\item \begin{equation*}[t_{ij}(\alpha),t_{kl}(\beta)]=\left\{\begin{array}{ll}
 t_{il}(\alpha\beta)  & \text{if } j=k \\
 t_{kj}(-\alpha\beta)  & \text{if } i=l \\
 I_n &\text{if } i\neq l, j\neq k
  \end{array}\right.
  \end{equation*}
for all $\alpha,\beta\in R$.
\item $d_{i}(\alpha)d_{i}(\beta)=d_{i}(\alpha\beta) , \forall \alpha,\beta \in \RM$
\item $[d_i(\alpha),d_j(\beta)]=I_n, \forall \alpha,\beta\in \RM$,

\item \begin{equation*}
d_k(\alpha^{-1})t_{ij}(\beta)d_k(\alpha)=\left\{\begin{array}{ll}
 t_{ij}(\beta)  & \text{if } k\neq i, k\neq j \\
 t_{ij}(\alpha^{-1}\beta) &\text{if } k=i\\
 t_{ij}(\alpha\beta)&\text{if } k=j
\end{array}\right.
\end{equation*}
\end{enumerate}
 \end{lemma}

\section{Abelian deformations of $T_n(R)$}\label{abdef:sec}
Even though abelian deformations of $T_n(R)$ are not necessarily matrix groups, they are close enough to justify the use of the same notation. For example we denote the identity element of such a group by $I_n$. In this section we first define abelian deformations $T_n(R,\bar{f})$ (already defined in Section~\ref{results:sec}) of $T_n(R)$ via generators and relations. 

For each $i=1, \ldots, n-1$ pick $f_i\in S^2(\RM,\RM)$. Define an \emph{abelian deformation} $T_n(R,f_1, \ldots, f_{n-1})$ of $T_n(R)$ as any group isomorphic to the group defined as follows.

$T_n(R,\bar{f})$ is the group generated by 
\[\{d_i(\alpha),t_{kl}(\beta): 1\leq i \leq , 1\leq k<l\leq n, \alpha\in \RM , \beta\in R\},\]

with the defining relations:
\begin{enumerate}
\item $t_{ij}(\alpha)t_{ij}(\beta)=t_{ij}(\alpha+\beta).$
\item \begin{equation*}[t_{ij}(\alpha),t_{kl}(\beta)]=\left\{\begin{array}{ll}
 t_{il}(\alpha\beta)  & \text{if } j=k \\
 t_{kj}(-\alpha\beta)  & \text{if } i=l \\
 I_n &\text{if } i\neq l, j\neq k
  \end{array}\right.
\end{equation*}
\item If $1\leq i\leq n-1$, then $d_{i}(\alpha)d_{i}(\beta)=d_{i}(\alpha\beta)diag(f_i(\alpha,\beta))$, where\\ $diag(f_i(\alpha,\beta))=d_1(f_i(\alpha,\beta))\cdots d_n(f_i(\alpha,\beta)), $
\item $[d_i(\alpha),d_j(\beta)]=I_n$
\item \begin{equation*}
d_k(\alpha^{-1})t_{ij}(\beta)d_k(\alpha)=\left\{\begin{array}{ll}
 t_{ij}(\beta)  & \text{if } k\neq i, k\neq j \\
 t_{ij}(\alpha^{-1}\beta) &\text{if } k=i\\
 t_{ij}(\alpha\beta)&\text{if } k=j
\end{array}\right.
\end{equation*}
\end{enumerate}
 

\begin{lemma}
\label{well-def:lem}
The set $T_n(R,\bar{f})$ is a group for any choice of $f_i\in S^2(\RM,\RM)$.
\end{lemma}
\begin{proof} 
This is really sketch of a proof. The $t_{ij}(\beta)$ generate a group $G_u\cong{UT_n(R)}$ by relations (1.) and (2.). The $d_i(\alpha)$ generate an abelian group $E_n$ by (3.) and (4.). Note that both of the above are closed under group operations and $G_u\cap E_n=I_n$. By (5.) $G_u$ is stable under the action of $E_n$ by conjugation which is described by (5.) itself, i.e. (5.) describes a homomorphism $\psi_{n,R}:E_n \to Aut(G_u)$ so that $T_n(R,\bar{f}) = E_n\ltimes_{\psi_{n,R}} G_u$, as an internal product, and $ker(\psi_{n,R})=Z(G)=\{diag(\alpha): \alpha\in \RM\}$.\end{proof}

For the future reference, if $R$ is a characteristic zero integral domain by \emph{$-I_n$} we denote the unique element of order 2 of $Z(T_n(R,\bar{f}))$. By $\pm I_n$ we mean $I_n$ or $-I_n$.

A few remarks are in order here.
\begin{rem}\label{biasNobias:rem}  Let $G=T_n(R,\bar{f})$ for some 2-cocycles $f_i\in S^2(\RM,\RM)$. The apparent bias in the definition of $T_n(R,\bar{f})$ to $j=1, \ldots ,n-1$ is no bias at all. Indeed from the definition it is clear that $G$ is a semi-direct product of the normal subgroup $G_u$ generated by $\{t_{ij}(\alpha): 1\leq i<j\leq n, \alpha \in R\}$ and the abelian subgroup $E_n$ generated by $\{d_i(\alpha):i=1,\ldots n, \alpha\in \RM\}$. Again it is clear that $G_u\cong UT_n(R)$.
So each 2-cocycle $f_i$, $i=1, \ldots , n-1$, defines the subgroup $d_i(\RM)$ generated by $\{d_i(\alpha):\alpha\in \RM\}\cup Z(G)$, as an abelian extension of $Z(G)\cong \RM$ by  $d_i(\RM)/Z(G)\cong \RM$. Let for $\alpha,\beta\in \RM$,
\[f(\alpha,\beta)=f_1(\alpha,\beta)\cdots f_{n-1}(\alpha,\beta),\]
and assume $f_n\in S^2(\RM,\RM)$ defines the subgroup $d_n(\RM)$ as an extension of $Z(G)$ by $\RM$. Then
\begin{align*}
diag(\alpha\beta)&=diag(\alpha)diag(\beta)\\
&= d_1(\alpha)d_1(\beta)\cdots d_n(\alpha)d_n(\beta)\\
&= diag(\alpha\beta)diag(f(\alpha,\beta)f_n(\alpha,\beta))
\end{align*}
So we can easily conclude that 
\[f_n(\alpha,\beta)=(f(\alpha,\beta))^{-1}.\]
 Now it is clear that given any $g\in S^2(\RM,\RM)$ one can set, say 
\[f_{n-1}(\alpha,\beta)=(f_1(\alpha,\beta))^{-1}\cdots (f_{n-2}(\alpha,\beta))^{-1}(g(\alpha,\beta))^{-1}\]
to see that $d_n(\RM)$ is defined now by the symmetric 2-cocycle $f_n=g$.    

 Finally we note (omitting the proof) that if $f_i\in B^2(\RM,\RM)$, for all $i=1,\ldots, n-1$, then $T_n(R)\cong T_n(R,\bar{f}).$ \end{rem}

\section{Nilpotent and Fitting subgroups}\label{fitt:sec}
 We start the section by some general remarks on Fitting subgroups and their first-order
definability in groups where the Fitting subgroup is nilpotent. Then we focus on the derived subgroup and Fitting subgroup of an abelian deformation of $T_n(R)$ where $R$ is a characteristic zero integral domain.
\subsection{Definability of the Fitting subgroup of a group where the Fitting subgroup is nilpotent}
The authors are indebted to V. A. Romankov for suggesting the proof of Lemma \ref{le:ncl-g}.

By the \emph{Fitting subgroup} of a group $G$, denoted by $Fitt(G)$ we mean the subgroup generated by all normal nilpotent subgroups of $G$.  Denote by $\CP$ the class of groups $G$  where the Fitting subgroup is itself nilpotent. For example every polycyclic-by-finite group is in $\CP$. Also $T_n(R)$ for any commutative associative ring $R$ unit is in $\CP$. Note that for every group $G$ in $\CP$, $Fitt(G)$ is the unique maximal normal nilpotent subgroup of $G$.

In this section we frequently use the following known result.

\begin{proposition} \label{pr:gamma-c}
Let $G$ be a group generated by a set $A$. Then for any $c \in \N$ the term $\gamma_cG$ of the lower central series of $G$ is generated as a normal subgroup by all left-normed commutators of length $c$ in the generators from $A$.
\end{proposition}

\begin{lemma} \label{le:nilpotent-c}
For any $c, n \in \N$ there is a formula $\Phi_c(x_1, \ldots,x_n)$  such that for any group $G$ and any tuple of elements $g_1, \ldots,g_n \in G$  the following holds:
\begin{equation} \label{eq:Phi-c}
G \models \Phi_c(g_1, \ldots,g_n) \Longleftrightarrow  \gamma_c\langle g_1, \ldots,g_n\rangle = 1.
\end{equation}
\end{lemma}
\begin{proof}
 Denote by $Com_c(x_1, \ldots,x_n)$ the set of all left-normed commutators of length $c$ in variables $x_1, \ldots,x_n$, viewed as words in the alphabet $x_1, x_1^{-1}, \ldots, x_n, x_n^{-1}$. Then by Proposition \ref{pr:gamma-c} the formula 
\[\Phi_c(x_1, \ldots,x_n) = \bigwedge_{w \in Com_c(x_1, \ldots,x_n)} w = 1\]
satisfies the condition (\ref{eq:Phi-c}). This proves the lemma.\end{proof}
\begin{corollary}\label{cor:nilpotent=c}
There is a formula $\Phi_{=c}(x_1, \ldots,x_n)$  such that for any group $G$ and any tuple of elements $g_1, \ldots,g_n \in G$  the formula 
$\Phi_{=c}(g_1, \ldots,g_n)$ holds in $G$ if and only if the subgroup generated by $  g_1, \ldots,g_n$ is nilpotent of class $c$.
\end{corollary}
\begin{corollary}\label{cor:nilpotent-maximal}
Assume that $g_1$, $g_2$, \ldots, $g_n$ are elements of a group $G$, which generate a maximal nilpotent subgroup $H$ of some nilpotency class, say $c$. Then

\[g\in H \Leftrightarrow G\models \Phi_{=c}(g,g_1, \ldots, g_n).\]
\end{corollary}

\begin{proof} The $\Rightarrow$ direction is clear. To prove the other direction just note that if $g$ satisfies $\Phi_{=c}(g,g_1, \ldots, g_n)$ then the subgroup $K=\langle g,g_1, \ldots, g_n\rangle$ is nilpotent and contains $H$ . So by maximality of $H$, $K=H$, which implies the result.
\end{proof}

\begin{lemma} \label{le:ncl-g} 
There is a formula $\Phi_{ncl,c}(x)$  such that for any group $G$ and any element $g \in G$  the formula 
$\Phi_{ncl, c}(g)$ holds in $G$ if and only if the normal subgroup generated by $g$ is nilpotent of class at most $c$.
\end{lemma}
\begin{proof}
The normal subgroup $\langle g \rangle^G$ generated by $g$ in $G$ is generated by a set $A = \{y^{-1}gy \mid y \in G\}$. We  mentioned  in the proof of Lemma \ref{le:nilpotent-c} that $\gamma_c\langle g \rangle^G = 1$  if and only if all the left-normed commutators of length $c$ in generators from $A$ are equal to 1 in $G$. Therefore the following formula does the job:
\[\Phi_{ncl,c}(x) = \forall y_1 \ldots \forall y_{c+1} ([y_1^{-1}xy_1, \ldots,y_{c+1}^{-1}xy_{c+1}] = 1.\] 
\end{proof}

A straightforward  modification of the argument in Lemma \ref{le:ncl-g}  proves the following lemma.
\begin{lemma}
There is a formula $\Phi_{ncl,c}(x_1, \ldots,x_n)$  such that for any group $G$ and any elements $g_1, \ldots,g_n \in G$  the formula 
$\Phi_{ncl,c}(g_1, \ldots,g_n)$ holds in $G$ if and only if the normal subgroup generated by $g_1, \ldots,g_n$ is nilpotent of class at most $c$.
\end{lemma}

Lemma \ref{le:ncl-g} immediately implies 
\begin{corollary}\label{le:FittingNilpotent}
Let $G$ be a group where  Fitting subgroup $Fitt(G)$ is nilpotent of class $c$. Then the formula $\Phi_{ncl,c}(x)$ (from Lemma \ref{le:ncl-g}) defines the subgroup $Fitt(G)$ in $G$.
 \end{corollary}

Now we show that nilpotency of Fitting subgroup of a group is an elementary property.

For $c, k \in \N$ define a sentence
\[
Fitt_{c,k} = \forall g (\Phi_{ncl,c+k}(g) \to \Phi_{ncl,c}(g))
\]
and put
\[
Fitt_c = \{Fitt_{c,k} \mid k \in \N\}.
 \]
It follows that if a group $G$ satisfies all the sentences from $Fitt_c$ then every nilpotent normal subgroup of the type $\langle g \rangle^G$ is nilpotent of class at most $c-1$.  Now define a sentence 
\[
\Phi_c ^\ast = \forall g_1 \ldots\forall g_c (\bigwedge_{i = 1}^n \Phi_{ncl,c}(g_i) \to \Phi_{ncl,c}(g_1, \ldots,g_n).
\]
and put 
\[
Fitt_c^\ast = \{Fitt_{c,k} \mid k \in \N\} \cup \{\Phi_c^\ast\}.
\]
The statement $\Phi_c ^\ast$ says that the set $A = \{ g \in G \mid  \gamma_c(\langle g \rangle^G) = 1\}$ generates a nilpotent group of class at most $c-1$.  

The argument above proves  
\begin{proposition}\label{pr:Fitt}
For any group $G$ the following holds
\[G \models Fitt_c ^\ast\Longleftrightarrow \gamma_cFitt(G) = 1.\]
\end{proposition}
\begin{corollary}\label{co:class-P}
The class $\CP$ is closed under elementary equivalence.
\end{corollary}


\subsection{ The derived subgroup and the Fitting subgroup of $T_n(R,\bar{f})$}
\begin{lemma}\label{G':lemma} Assume $R$ is a commutative associative ring with unit. Then the derived subgroup $G'$ of $G=T_n(R,\bar{f})$ is the subgroup of $G$ generated by
\[X=\{t_{i,i+1}((1-\alpha)\beta), t_{kl}(\beta):1\leq i\leq n-1, 1< k+1<l\leq n, \alpha\in\RM, \beta\in R\}.\]\end{lemma} 

\begin{proof}
Let $N$ denote the subgroup generated by $X$. Each $t_{kl}(\beta)$, with $l-k\geq 2$ is already a commutator by definition, and \[d_i(\alpha^{-1})t_{i,i+1}(-\beta)d_{i}(\alpha)t_{i,i+1}(\beta)=t_{i,i+1}((1-\alpha^{-1})\beta),\]for any $\alpha\in \RM$ and $\beta\in R$, hence $N\leq G'$. To prove the reverse inclusion firstly note that since $G/UT_n(R)$ is abelian, $G'\leq UT_n(R)=G_u$. Now Pick $x,y\in G$ then $x=x_1x_2$ and $y=y_1y_2$ where $x_1,y_1\in D_n(R,\bar{f})$ and $x_2,y_2\in UT_n(R)$. Now 
\begin{align*}[x,y]&=[x_1x_2,y_1y_2]\\
&=[x_1,y_1]^{z_1}[x_1,y_2]^{z_2}[x_2,y_1]^{z_3}[x_2,y_2]^{z_4},\\
&= [x_1,y_2]^{z_2}[x_2,y_1]^{z_3}[x_2,y_2]^{z_4}\end{align*}
for some $z_i\in G$, $i=1, \ldots 4$. The commutator $[x_2,y_2]\in (G_u)'$, where $(G_u)'$ is characteristic in $G_u=UT_n(R)$ so normal in $G$. Therefore $[x_2,y_2]^{z_4}$ is a product of $t_{ij}(\beta)$, $i+1<j$. The commutators $[x_2,y_1]$ and $[x_1,y_2]$ are of the same type. So let us analyze one of them. Indeed $x_2=d_1(\alpha_1)\cdots d_n(\alpha_n)$ and $y_1=t_{12}(\beta_{12})\cdots t_{1n}(\beta_{1n})$. So indeed $[x_2,y_1]$ is a product of conjugates of commutators of type $[d_k(\alpha),t_{ij}(\beta)]$. In case that $j>i+1$ this is conjugate of a $t_{ij}(\beta)\in G_u'$ which was dealt with above and is an element of $N$. It remains to analyze the conjugates of $t=t_{i,i+1}((\alpha-1)(\beta))$. Consider $z=xy$, $x=d_1(\alpha_1)\cdots d_n(\alpha_n)\in E_n$ , $y\in G_u$. Then $t^x= t_{i,i+1}((\alpha-1)\alpha_i^{-1}\alpha_{i+1}\beta) \in N$ and $N$ is normalized by $y$ anyway. This completes the proof.       
\end{proof}

\begin{lemma}\label{utn:defn:lem} If $G=T_n(R)$ where $R$ is a commutative associative ring with unit, then the derived subgroup $G'$ of $G$ is uniformly definable in $G$.\end{lemma}
\begin{proof} Recall the description of $G'$ from Lemma~\ref{G':lemma}. Note that every element of $G'$ can be written in a unique was as a product of transvections. This means that $G'$ is a verbal subgroup of finite width, exactly $M=n(n-1)/2$. Therefore $G'$ is definable in $G$ uniformly with respect to $Th(G)$ by the $L$-formula, 
\begin{equation} \label{PhiG':eqn} \Phi_{G'}(x)\define\exists x_1, \ldots, x_M,y_1, \ldots, y_M ( x=[x_1,y_1]\cdots [x_M,y_M]).\end{equation}
The same $L$-formula clearly defines $H'$ in an elementary equivalent copy $H$ of $G$. \end{proof}
\begin{lemma}\label{Fitt-desc:lem} For $G= T_n(R,\bar{f})$, $Fitt(G)=UT_n(R)\cdot Z(G)$.\end{lemma}
\begin{proof}
Clearly $Fitt(G)\geq UT_n(R)\cdot Z(G)$. Assume $N$ is a normal nilpotent subgroup of $G$. Then the product $M$ of $N$ and $UT_n(R)\cdot Z(G)$ is also normal in $G$ and nilpotent. If $M$ has an element not already contained in $UT_n(R)\cdot Z(G)$ then $M$ has to contain some $d_1(\alpha_1)\cdots d_n(\alpha_n)$ where for a pair $1\leq i\neq j\leq n$, $\alpha_i\neq \alpha_j$. Pick such an element, say $x$. Then
\begin{align*}
x^{-1}t_{ij}(-\beta)x&= d_n(\alpha^{-1}_n)\cdots d_1(\alpha_1^{-1})t_{ij}(-\beta)d_1(\alpha_1)\cdots d_n(\alpha_n)\\
&=d_j(\alpha^{-1}_j) d_i(\alpha_i^{-1})t_{ij}(-\beta)d_i(\alpha_i) d_j(\alpha_j)\\
&=t_{ij}(-\alpha_i^{-1}\alpha_j\beta)
\end{align*} 
So $[x,t_{ij}(\beta)]=t_{ij}(\beta(-\alpha^{-1}_i\alpha_j+1))$. If $\beta\neq 0$ the assumption that $\alpha_i\neq \alpha_j$ and $R$ being an integral domain shows that $[x,t_{ij}(\beta)]\neq I_n$. A simple inductive argument shows that for any positive integer $m$
\[[\underbrace{x,\ldots,x}_{m-\text{times}},t_{ij}]\neq I_n,\]
which contradicts the nilpotency of $M$. So $Fitt(G)=UT_n(R)\cdot Z(G)$. 
\end{proof}

\section{Proof of Theorem~\ref{mainthm}.}\label{mainthm:sec}
We begin the proof of Theorem~\ref{mainthm} by proving a few auxiliary lemmas and reviewing some crucial results on the model theory of $UT_n(R)$.  

\begin{lemma}[\cite{beleg99}, Proposition 1.7.1]\label{utn-coord:lem} Consider the group $N=UT_n(R)$, where $R$ is a commutative associative ring with unit. Then for each $1\leq i<j \leq n$ one-parameter subgroups $T_{ij}=\{t_{ij}(\alpha):\alpha\in R\}$ are definable in $N$, unless $j=i+1$, Indeed the ring $R$ and its action on each such $t_{ij}$ is interpretable in $N$ (with respect to the constants $\bar{t}$). If $j=i+1$ then the subgroup $C_{ij}=T_{ij} \cdot Z(N)$ is definable in $N$.\end{lemma} 
In general the fact that $C_{i,i+1}=T_{i,i+1}\cdot Z(N)$  from Lemma~\ref{utn-coord:lem} is a split abelian extension of $Z(H)\cong R^+$ by $C_{i,i+1}/Z(G)\cong R^+$ is not a first-order property. As O.V. Belegradek shows in \cite{beleg99} a group elementarily equivalent to a $UT_n(R)$ is almost isomorphic to a $UT_n(S)$ for some $S\equiv R$ except that in $H$, $C_{i,i+1}$ might be isomorphic to a non-split extension of $S^+$ by $S^+$. Such a group is called a \emph{quasiunitriangular} group and is denoted by $UT_n(S,g_1, \ldots, g_{n-1})$ for symmetric 2-cocyles $g_i\in S^2(S^+,S^+)$.    
\begin{thm}[\cite{beleg99}, Proposition 1.8.1 and 2.2.7]\label{utn-corrdinatization:thm}Let $R$ be a commutative associative ring with unit. Then, there is a formula $ \Phi_{UT_n}(\bar{x})$ of $L$ which holds on the tuple of  transvections $t_{ij}$, with some specific ordering, in $N=UT_n(R)$ and if $H\equiv N$ is a group, then $H\models \exists \bar{x}\Phi_{UT_n}(\bar{x})$ implies that $H\cong UT_n(S,\bar{g})$ for some ring $S\equiv R$ and some symmetric 2-cocycles $g_i\in S^2(S^+,S^+)$.\end{thm}

We recall that $-I_n$ denotes the unique element of order 2 of $Z(T_n(R,\bar{f}))$ when $R$ is a characteristic zero integral domain.

\begin{lemma}\label{Gupm:lem} Assume $R$ is a characteristic zero integral domain, $G=T_n(R)$. Then  there is an $L$-formula $\Phi_{G_{u^\pm}}(x)$ that defines $G_{u^\pm}=UT_n(R)\times \ZZ $ in $G$ if $G'\neq UT_n(R)$. Otherwise $G_u=UT_n(R)$ is definable in $G$ as $G'$. In any case the definitions are uniform with respect to $Th(G)$. Hence the subgroups $T_{ij}$, $j-i\geq 2$ are definable in $G$ relative a constant $t_{ij}$. IF $G'=G_u$ the subgroups $C_{i,i+1}$ are definable in $G$, otherwise, the subgroups $\pm C_{i,i+1}=\ZZ\cdot C_{i,i+1}$ are definable in $G$, in both cases relative to the constants $t_{i,i+1}$. \end{lemma}
\begin{proof}
To prove (a) recall that $G'$ and $Fitt(G)$ are both uniformly definable in $G$. Also recall the descriptions of those subgroups from Lemma~\ref{G':lemma} and~\ref{Fitt-desc:lem}. Either $G'=G_u$ and we let $\Phi_{G_{u^\pm}}=\Phi_{G'}$. Or $G'< G_u$ and we let
\begin{equation}
\label{PhiGupm:eqn} \Phi_{G_{u^\pm}}(x)\define x\in Fitt(G) \wedge ( x\in G' \vee (x\notin G' \wedge x^2\in G')).
\end{equation} The rest of the statement follows from Lemma~\ref{utn-coord:lem}. 
\end{proof}
\begin{lemma}\label{(d,t):lem} Assume $R$ is a characteristic zero integral domain and $G=T_n(R)$. There exists an $L$-formula $\Phi_d(\bar{x},\bar{y})$ which holds on $(\bar{d},\bar{t})$ in $G$ and $\Phi_d(\bar{d},\bar{t})$ expresses that: 
\begin{enumerate}
\item $\ds( \bigwedge_{1\leq i< j \leq n}\Phi_{G_{u^{\pm}}}(t_{ij})) \wedge \Phi_{UT_n}(\bar{t}\ZZ)$, where $\bar{t}\ZZ$ denotes the ordered tuple of the cosets $t_{ij}\ZZ$.
\item  $d_k^2=1$, for all $1\leq k \leq n$,    
\item $[d_k,d_l]=1$, for all $1\leq k,l\leq n$
\item $d_1\cdots d_n=-I_n\in Z(G)$, where $-I_n$ is the unique element of order 2 in $Z(G)$
\item 
$\forall x\in T_{ij}, 1\leq i+1 < j\leq n$
\begin{equation*}\label{conj-actx:eq1b}
 d_k xd_k=\left\{\begin{array}{ll}
 x  & \text{if } k\neq i, j \\
 x^{-1} &\text{if } k=i \text{ or } k=j 
\end{array}\right.
\end{equation*}
\item for all $i=1,\ldots, n-1$
and $\forall x\in \pm C_{i,i+1},  \exists!~ \delta(x,d_k)\in Z(G')$ such that
\begin{equation*}\label{conj-actx:eq2}
 d_k xd_k=\left\{\begin{array}{ll}
 x\delta(x,d_k)  & \text{if } k\neq i,i+1 \\
 x^{-1}\delta(x,d_k) &\text{if } k=i \text{ or } k=i+1 
\end{array}\right.
\end{equation*}
\item  The subgroup defined by
\[\pm T_{12}\define \{x\in \pm C_{12}: x^2\in [d_2,\pm C_{12}]\},\]
satisfies: $ \pm C_{12} = \pm T_{12} \cdot Z(G')$, and $\pm T_{12} \cap Z(G')=I_n.$
\item for all $2\leq i \leq n-1$, the subgroup defined by
\[\pm T_{i,i+1}\define \{x\in \pm C_{i,i+1}: x^2\in [d_i,\pm C_{i,i+1}]\},\]
satisfies: $ \pm C_{i,i+1} = \pm T_{i,i+1} \cdot Z(G')$ , and $\pm T_{i,i+1} \cap Z(G')=I_n.$  
\end{enumerate}
\end{lemma}
\begin{proof} The formulas in (1) are taken from Theorem~\ref{utn-corrdinatization:thm} and Lemma~\ref{Gupm:lem}. The fact that the $d_i$ satisfy (2)-(6) is clear. All the subgroups $Z(G)$, $Z(G')$, $T_{ij}$, $j-i\geq 2$ and $\pm C_{i,i+1}$ are definable in $G$ as observed above, So indeed (2)-(6) are fist order properties which hold on the $d_i$. Statements (7) and (8) are similar, so let us consider (7) only. Firstly, $[d_2,\pm C_{12}]$ is definable in $G$ relative to the constants $d_2$ and $t_{12}$ by the formula $\exists y\in C_{12}(x=[d_2,y])$. Secondly, an element of $\pm C_{12}$ is of the form $y=\pm I_n t_{12}(\alpha)t_{1n}(\gamma)$, therefore
\begin{equation*}\label{[d2c12]:eqn}\begin{split}
[d_2, y] &=[d_2(-1), \pm I_n t_{12}(\alpha)t_{1n}(\gamma)]\\
&=d_2(-1)t_{12}(-\alpha)t_{1n}(-\gamma)d_2(-1)t_{12}(\alpha)t_{1n}(\gamma)\\
&=t_{12}(\alpha)t_{1n}(-\gamma)t_{12}(\alpha)t_{1n}(\gamma)\\
&=t_{12}(2\alpha)=(t_{12}(\alpha))^2
\end{split}\end{equation*}
which shows that $[d_2,\pm C_{i,i+1}]=\{ t_{12}(\alpha)^2=t_{12}(2\alpha): \alpha \in R\}=T_{12}^2$. So it is clear that if $x^2\in T_{12}^2$ for $x\in \pm C_{i,i+1}$, then $x\in \pm T_{12}$, for $\pm T_{12}$ defined as above. The two facts
$ \pm C_{i,i+1} = \pm T_{i,i+1} \cdot Z(G')$ , and  $\pm T_{i,i+1} \cap Z(G')=I_n$, which clearly hold in $G$ are first-order properties since all subgroups involved are definable in $G$.\end{proof}

\begin{proposition}\label{TijDefine:prop} Assume $G$ and $R$ are as in Lemma~\ref{(d,t):lem} and $H\equiv G$ as groups, where (to consider the most general case) $G'$ is a proper subgroup of $G_u$. If $(\bar{e},\bar{s})$ are elements of $H$ where $H\models\Phi_d(\bar{e},\bar{s})$ then there exists a ring $S\equiv R$ and formulas $\Phi_{ij}(x,\bar{e}, \bar{s})$ such that: 
\begin{enumerate}
\item If $1< i+1 <j\leq n$ each $\Phi_{ij}(H,\bar{e},\bar{s})$ is a one-parameter subgroup of $H$ generated by $s_{ij}$ in $H$ over $S$.
\item If $j=i+1$ then $\Phi_{ij}(H,\bar{e},\bar{s})$ is a subgroup $\pm S_{i,i+1}$ in $H$, which is an abelian extension of $\ZZ$ by the one-parameter subgroup generated by the coset $s_{i,i+1}\ZZ$ over $S$. Indeed $\pm S_{i,i+1}\cong S_{i,i+1}\times \ZZ$. 
\item The formula $\Phi_{G_{u^\pm}}(x)$ defines in $H$ a subgroup $H_{u^\pm}$,
  where $H_{u^\pm}\cong UT_n(S)\times \ZZ$.  
\end{enumerate}
\end{proposition}

\begin{proof} Statement (1) already follows from Belegradek's work, say Theorem~\ref{utn-corrdinatization:thm} and uniform definability of $G'$.  
	
For (2), the first-oder definability $\pm T_{i,i+1}=\{\pm I_nt_{i,i+1}(\beta):\beta \in R\}$  was proved in Lemma~\ref{(d,t):lem}. Moreover, by Theorem~\ref{utn-corrdinatization:thm} the formulas that define the $\pm C_{i,i+1}/\ZZ$ in $G_{u^\pm}/\ZZ$ define the subgroups $\pm D_{i,i+1}/{\ZZ}$ in $H_{u^\pm}/\ZZ$ such that 
$\pm D_{i,i+1}/Z(H_{u^\pm}) \cong S^+$, and $Z(H_{u^\pm})/\ZZ\cong S^+$.
This, together with part (7) of Lemma~\ref{(d,t):lem}, and definability of $\ZZ$ imply the existence of the formula above defining $\pm{S_{i,i+1}}$ as claimed. The fact that $\pm S_{i,i+1}= \ZZ \times S_{i,i+1}$ is implied by the fact that $Ext(S^+, \Z_2)=0$, since $S$ is a characteristic zero integral domain (See Lemma~\ref{tbytf:lem} stated and proven below).

Now (1), (2) and Theorem~\ref{utn-corrdinatization:thm} clearly imply (3).

     \end{proof} 
 \begin{lemma}
 \label{tbytf:lem}
 Assume $B$ is a torsion-free abelian group and $T$ is a finite abelian group, both written additively. Then $Ext(B,T)=0$.\end{lemma}
 \begin{proof} Assume $C_m$ is the cyclic group of order $m$. Then by~(\cite{fox}, 52.F) we have that,  
  $Ext(B,C_m)\cong Ext(B[m],C_m)$ where $B[m]$ denotes the subgroup of $B$ consisting of all $b\in B$ such that $mb=0$. But since $B$ is torsion-free $B[m]=0$, which implies that $Ext(B,C[m])=0$. 
 Now assume $T\cong C_{m_1}\oplus \cdots \oplus C_{m_k}$ is the primary decomposition of $T$ into finite cyclic groups $C_{m_i}$. Then, 
 \[Ext(B,T)\cong \bigoplus_{i=1}^k Ext(B,C_{m_i})=0.\] \end{proof} 
 The following lemma follows from standard linear algebra arguments. 
 \begin{lemma}\label{Dn:lem}Let $R$ be a characteristic zero integral domain. Then the subgroup $D_n(R)$ of all diagonal matrices of $T_n(R)$ is definable in it by the first formula 
 \[\Phi_{D}(x,\bar{d})\define(\bigwedge_{i=1}^n[x,d_i]=1),\]
with constants $d_i=d_i(-1)$, $1\leq i \leq n$, i.e. $D_n(R)$ is the centralizer in $G$ of all the $d_i$. \end{lemma}

 
\begin{corollary}\label{char:firsttry:cor}
Assume $H\equiv T_n(R)$, where $R$ is a characteristic zero integral domain. Then \[H\cong E_n \ltimes UT_n(S),\] where $E_n$ is an abelian subgroup of $H$, defined by the same formula, relative to some constants $e_1, \ldots, e_n$, that defines $D_n$ in $G$, and $S\equiv R$.
  \end{corollary}
\begin{proof}We observed that $G'$ and $G_{u^\pm}=UT_n(R)\cdot \ZZ$ are definable in $G$ uniformly with respect to $Th(G)$ and constants $(\bar{t},\bar{d})$ satisfying $\Phi_d(\bar{t},\bar{d})$. By Proposition~\ref{TijDefine:prop} the same formula that defines $G_{u^{\pm}}$ defines in $H$, $H_{u^{\pm}}=H_u\cdot\ZZ$, where $H_u\cong UT_n(S)$. The subgroup $H_u$ may not be a definable  subgroup, so we can not immediately conclude that $H_u$ is a normal subgroup of $H$ even though $H_{u^{\pm}}$ is so. Consider the following sentence for each $i$:
\[\Phi_{iN}\define \forall x\in \pm T_{i,i+1},\forall y\in G, \exists z\in G'(x^y= xz  \vee x^y=x^{-1}z).\] 
All subgroups involved are definable and so the statement is an $L$-sentence. Let us show that it is true in $G$. Without loss of generality we can assume $x= t_{i,i+1}(\beta)$ and $y= d_{i}(\alpha)d_{i+1}(\gamma)$. Then
$x^y=t_{i,i+1}(\alpha^{-1}\gamma\beta)$. So either $\alpha=\gamma=1$ and $x^y=x$ or $\alpha^{-1}\gamma\neq 1$ and  
$x^yx^{-1}=t_{i,i+1}((\alpha^{-1}\gamma-1)\beta)\in G'$. Therefore the fact that $T_{i,i+1}$ is normalized modulo $G'$ is a first-order property. By uniformity of definitions and the fact that $H\models \Phi_{iN}$ for each $i$ we may conclude that for each $i$, $S_{i,i+1}$ is normalized in $H$ modulo $H'$. Therefore $H_u\geq H'$ is normal in $H$ modulo $H'$. But this obviously implies that $H_u\cong UT_n(S)$ is normal in $H$.   

 Now by Proposition~\ref{TijDefine:prop} and Lemma~\ref{Dn:lem} we can express the facts that $G=D_n\cdot G_{u^\pm}$, $D_n\cap G_{u^\pm}=\ZZ$ and  $G_{u^\pm}\unlhd G$ using $L$-formulas uniformly with respect to $Th(G)$ and constants $\bar{d}$ and $\bar{t}$ satisfying $\Phi_d(\bar{d},\bar{t})$. Considering that  $H_{u^\pm}=H_u \times \ZZ$, $H_u\cong UT_n(S)$, $H_u\unlhd H$, and the fact that $D_n$ is abelian, there exists an abelian subgroup $E_n$ of $H$ such that $H\cong E_n\ltimes UT_n(S)$ as claimed.\end{proof} 
In the next statement we clarify the structure of the subgroup $E_n$ introduced above.

\begin{lemma}\label{torus:lem} Let $G=T_n(R)$, where $R$ is characteristic zero integral domain and let $H\equiv G$ as groups. Assume 
\[H\models (\bigwedge_{1\leq i<j\leq n}\Phi_{G_{u^\pm}}(s_{ij}))\wedge \Phi_{UT_n}(\bar{s})\wedge \bigwedge_{i=1}^n\Phi_d(\bar{e},\bar{s})\] and $E_n$ is an abelian subgroup of $H$ defined by $\Phi_D(x,\bar{e})$. Then:

\begin{enumerate}
\item[(a)] For each $1\leq i \leq n$ the subgroup  $\Delta_i(R)\define d_i(\RM) \cdot Z(G)$ is first-order definable in $D_n=D_n(R)$ by a formula $\Phi_{\Delta_i}(x,\bar{d},\bar{t})$. 
Moreover there exists a ring $S\equiv R$ and for each $i=1,\ldots, n$ a subgroup $\Lambda_i<E_n$ of $H$ such that
\[H \models \Phi_{\Delta_i}(x,\bar{e},\bar{s}) \Leftrightarrow \left((x \in \Lambda_i) \wedge (Z(H)<\Lambda_i) \wedge  (\Lambda_i/Z(H)\cong \SM)\right)\]
\item[(b)] $D_n=\Delta_1\cdots \Delta_n$. Therefore $E_n=\Lambda_1\cdots \Lambda_n$. 
\item[(c)] $\ds Z(G)=\bigcap_{i=1}^n \Delta_i$ and $Z(G)$ is definably isomorphic to $\RM$. Similarly one has $\ds Z(H)=\bigcap_{i=1}^n \Lambda_i$ and $Z(H)\cong \SM$.
\item[(d)] $E_n$ is isomorphic to an abelian extension of $Z(H)\cong \SM$ by $E_n/Z(H)\cong (\SM)^{n-1}$.
\end{enumerate}\end{lemma}

\begin{proof}  Pick a $1\leq k \leq n$, say $k=1$. Then
\begin{equation} \label{Delta1-defn:eqn} x\in \Delta_1 \Leftrightarrow x\in D_n \wedge \exists!\alpha\in \RM, \forall \beta\in R: x^{-1}t_{ij}(\beta)x=\left\{ \begin{array}{cc}
t_{ij}(\alpha^{-1}\beta)& \text{if~} i=1\\
t_{ij}(\beta) & \text{if~} i\neq 1 
\end{array}  \right.\end{equation}
By Proposition~\ref{TijDefine:prop} and Lemma~\ref{Dn:lem} the right-hand side of the above equivalence is expressible using $L$-formulas. Just an explanation is in order here for the cases when $j=i+1$. Recall form proof of Corollary~\ref{char:firsttry:cor} that $x^{-1}(\pm I_n t_{i,i+1}(\beta))x=\pm I_n t_{ij}(\alpha\beta)$ indeed implies that $x^{-1}(t_{i,i+1}(\beta))x=t_{ij}(\alpha\beta)$. Note that if $x,y\in D_n$ both satisfy the right hand side of~\eqref{Delta1-defn:eqn} for the same $\alpha\in \RM$ then the actions of $x$ and $y$ on all elements of $G$ by conjugation are the same and therefore $xy^{-1}\in Z(G)$. This proves that $\Delta_i/Z(G)$ and $\RM$ are definably isomorphic and also that each $\Delta_i$ is definable in $D_n$ and therefore in $G$. By Corollary~\ref{char:firsttry:cor} and since $E_n$ is defined in $H$ by the same formula that defines $D_n$ and the uniformity of the action of the ring $R$ on the $t_{ij}$  the statement regarding the $\Lambda_i$ follows immediately. Parts (b) follows easily from (a). For (c) similar to (a) we can write an $L$-formula expressing that
\begin{align*} x\in Z(G)  \Leftrightarrow &x\in D_n \wedge \exists!\alpha\in \RM, \exists x_1\in \Delta_1, \ldots, \exists x_n\in\Delta_n, \forall \beta\in R\\ 
&x=x_1\cdots x_n  \wedge x_k^{-1}t_{ij}(\beta)x_k=\left\{ \begin{array}{ll}
t_{ij}(\alpha^{-1}\beta)& \text{if~} k=i\\
t_{ij}(\alpha\beta)& \text{if~}k=j\\
t_{ij}(\beta) & \text{if~} k\neq i,j 
\end{array}  \right.\end{align*}
and some basic properties of $T_n(R)$. Part (d) follows from (a)-(c). \end{proof}

\emph{Proof of Theorem~\ref{mainthm}.} We use the notation of Lemma~\ref{torus:lem}. Note that $\Lambda_i=\langle e_i(\alpha): \alpha \in \SM\rangle \cdot Z(H)$ where $e_i(\alpha)$ is the element whose action on the $s_{ij}(\beta)$ is precisely the same as those of $d_i(\alpha)$ on the $T_{ij}$ for $T_n(S)$. Assume that each $\Lambda_i$, $i=1, \ldots , n-1$ is defined by the 2-cocycle $f_i\in S^2(\SM,\SM)$as an extension of $Z(H)\cong \SM$ by $\Lambda_i/Z(H)\cong \SM$. Then it is clear that the $s_{ij}(\beta)$ and the $e_i(\alpha)$ generate $H$ and satisfy the exact relations (1)-(5) in the definition of $T_n(S,\bar{f})$ given at the beginning of Section~\ref{abdef:sec}. \qed

\section{Which abelian deformations of $T_n(R)$ are elementarily equivalent to it?}\label{suff:sec}

Theorem~\ref{mainthm} gives a necessary condition for a group $H$ to be elementarily equivalent to $T_n(R)$ for a characteristic zero integral domain $R$. In some cases the sufficient condition is actually a bit stronger than the one found in the referred theorem. In this section we prove a sufficient condition for certain classes of integral domains.

Let us recall that for ring $R$ a  2-cocycle $f:\RM\times \RM \to \RM$ is said to be \emph{coboundarious on torsion} or $\emph{CoT}$ if the restriction $g:T\times T \to \RM$ of $f$ to $T\times T$ is a 2-coboundary, where $T=T(\RM)$. Assume $A$ is an abelian extension of $A_1=\RM$ by $A_2=\RM$, and $T_2$ is the copy of $T$ in $A_2$. Then $f$ is CoT if and only if the subgroup $H$ of $A$ generated by $A_1$ and any preimage of $T_2$ in $A$ splits over $A_1$, i.e. $H\cong A_1\times T_2$. We will note later what is the rationale behind this definition.

\begin{rem}\label{split:rem}Assume for an abelian group $A$ we have $A\cong T \times B$ where $T$ and $B$ are some subgroups of $A$. Consider a symmetric 2-cocycle $f: A \to A$. By abuse of notation we consider $f$ as $f: T \cdot B \to T \times B$. Then $f$ is cohomologous to $(g_1g_2, h_1h_2)$ where $g_1\in S^2(T,T))$, $g_2\in S^2(T,B)$, $h_1\in S^2(B,T)$ and finally $h_2\in S^2(B,B)$. We will use this notation in the following. 

\end{rem}
We will need to state a few well-known definitions and results.
 
Let $B$ be an abelian group and $A$ a subgroup of $B$. Then $A$ is called a \emph{pure subgroup of $B$} if $\forall n\in \mbb{N}$, $nA=nB\cap A$.
\begin{lemma}\label{pure:lem}Let $A\leq B$ be abelian groups such that the quotient group $B/A$ is torsion-free. Then $A$ is a pure subgroup of $B$.\end{lemma}
\begin{proof}One direction is trivial. For the other direction assume that $g\in nB\cap A$. Then there is $h\in B$ such that $g=nh$. to get a contradiction assume that $h\notin A$. Then $g=nh\notin A$ since $B/A$ is torsion free. A contradiction! So $h\in A$, therefore $g=nh\in nA$.\end{proof}

An abelian group $A$ is called \emph{pure-injective} if $A$ is a direct summand in any abelian group $B$ that contains $A$ as a pure subgroup.

The following theorem expresses a connection between pure-injective groups and uncountably saturated abelian groups.
\begin{thm}[\cite{eklof}, Theorem 1.11]\label{ekthm} Let $\kappa$ be any uncountable cardinal. Then any $\kappa$-saturated abelian group is pure-injective.\end{thm} 

\begin{rem}
\label{finducef*:rem}
Assume $A$ and $B$ are abelian group and $f\in S^2(B,A)$. Let $\D$ be an ultrafilter on a set $I$. Let $A^*$ and $B^*$ denote the ultrapowers of $A$ and $B$, respectively, over $(I,\D)$. Then $f$ induces a natural 2-cocycle $f^*\in S^2(B^*,A^*)$ representing an abelian extension of $A^*$ by $B^*$ (See Lemma 7.1 of~\cite{MS2009} for details)). \end{rem}
\begin{lemma}\label{saturated-split:lem} Assume $R$ is a characteristic zero integral domain so that the maximal torsion subgroup of $\RM$ is finite. Assume $f\in S^2(\RM, \RM)$ is CoT and $(I,\D)$ is an ultra-filter so that ultraproduct $(\RM)^*$ of $\RM$ over $\D$ is $\aleph_1$-saturated. Then the 2-cocycle $f^* \in S^2((\RM)^*,(\RM)^*)$ induced by $f$ is a 2-coboundary.\end{lemma}
\begin{proof}
Firstly note that $(\RM)^*=(R^*)^\times$ and $T((\RM)^*)=T^*\cong T$. Also $(\RM)^*/T^*$ is torsion-free or trivial. Assume it is not trivial. Since $\RM$ is infinite then $(\RM)^*$ is $\aleph_1$-saturated. Since $T^*$ is finite abelian and $(\RM)^*/T^*$ is torsion-free by Lemma~\ref{tbytf:lem} there is a subgroup, say $B^*$ of $(\RM)^*$ such that $(\RM)^*= T^*\times B^*$ as an internal direct product. The assumption that $f$ is CoT implies that the induced cocycle $f^*\in S^2((\RM)^*,(\RM)^*)$ is CoT. Then, using the notation of Remark~\ref{split:rem} we have that $g_1^*\in S^2(T^*,T^*)$ and $g^*_2\in S^2(T^*,B^*)$ are both coboundaries. The 2-cocycle  $h_1^*\in S^2(B^*,T^*)$ is a coboundary since $Ext(B^*,T^*)=1$ by Lemma~\ref{tbytf:lem}. The 2-cocycle $h_2^*\in S^2(B^*,B^*)$ is also a coboundary since $B^*$ is pure-injective by Theorem~\ref{ekthm} and also $B^*$ is a pure subgroup in an extension represented by $h_2^*$ because of Lemma~\ref{pure:lem}. Consequently $f^*$ is a coboundary.
\end{proof}     

\emph{Proof of Theorem~\ref{main(improved):thm}. }
The proof of $\Rightarrow$ direction is mostly included in the proof of Theorem~\ref{mainthm}. We just need to verify that each $f_i$ is in addition CoT. Again recall that $f_i$ corresponds to the extension $E_i=D_i(S,f_i)$ of $Z(H)\cong\SM$ by $E_i/Z(H)\cong\SM$. By assumption $D_i(R)/Z(G)$ has a finite maximal torsion subgroup, say $T$ of cardinality $c$. Since $D_i(R)/Z(G)\equiv D_i(S)/Z(H)$, $D_i(S)/Z(H)$ has the same (up to isomorphism) finite subgroup $T$ as the maximal torsion subgroup. So indeed there exists an $L$-formula $\Phi_T(x_1,\ldots x_c)$ describing the multiplication table of $T$. Now consider  the sentences 
\[\Theta_i\define \exists x_1, \ldots, \exists x_c \left( (\bigwedge_{j=1}^c x_j\in D_i)\wedge (x_j\notin Z(G) \vee x_j=1) \wedge \Phi_T(\bar{x})\right).\] 
Each $\Theta_i$ is clearly expressible in $L$. In plain language what they say is that the subgroup $D'_i$ of $D_i$ generated by $Z(G)$ and any representatives of elements of $T$ splits over $Z(G)$ which is true in $G$. Since $\Theta_i$ hold in $H$ too we make the same conclusion in $H$. That is, to say the $f_i\in S^2(\SM,\SM)$ need to be CoT. 

To prove $\Leftarrow$ direction we notice that $T_n(R)$ is interpretable in $R$ uniformly with respect to $Th(R)$. Indeed every matrix in $T_n(R)$ can be perceived as an element of $(R^*)^n\times R^{n(n-1)/2}$ which is a definable in $R$ and vice versa. The product and inversion on these tuples are computed by the polynomials that define matrix multiplication.  So $T_n(R)\equiv T_n(S)$ for any $S\equiv R$.

To conclude the proof we need to prove that $T_n(S)\equiv T_n(S,\bar{f})$ where each $f_i$ is CoT. Let $(I,\D)$ be an $\aleph_1$-incomplete ultrafilter. As usual, by $C^*$ we mean the ultrapower $C^I/\D$ of a structure $C$. Then $B((S^*)^\times)=B((\SM)^*)=B^*(\SM)$ is either trivial or $\aleph_1$-saturated. If $f^*_i\in S^2((\SM)^*,(\SM)^*)$ denotes the 2-cocycle induced by $f_i$ then for each $i=1, \ldots,n-1$, $f^*_i$ is a 2-coboundary by Lemma~\ref{saturated-split:lem}. The fact that $T^*_n(S,\bar{f})\cong T_n(S^*,\bar{f^*})$ requires only some routine checking. Therefore \[T^*_n(S,\bar{f})\cong T_n(S^*, \bar{f^*})\cong T_n(S^*)\cong T_n^*(S).\]  
This concludes the proof utilizing Keisler-Shelah's theorem.\qed

\section{The case of $T_n(F)$ where $F$ is an algebraically closed of characteristic zero or a real closed field} \label{closed:sec}

\emph{Proof of Theorem~\ref{charthm-alg-realclosed:thm}. }By Theorem~\ref{mainthm}, there exist $K\cong F$ and 2-cocycles $f_i\in S^2(\KM,\KM)$ such that $H\cong T_n(K,\bar{f})$.

If $F$ is algebraically closed so is $K$. Hence $Z(H)\cong \KM$ is divisible and it's a direct factor in every abelian group it is a subgroup of. So $Ext(\KM,\KM)=1$ and therefore $T_n(K)\cong T_n(K,\bar{f})$ for any choice of the 2-cocycles $f_i$.

Now assume $F$ is a real closed field. Therefore $K$ is a real closed field. For any real closed field $K$, $\KM=A \times B(K)$ where  $A=\{\pm 1\}$ and $B(K)$ is a torsion-free divisible group. 
Consider $\Lambda_i(K)$, from Lemma~\ref{torus:lem}, which is an abelian extension of $Z(H)\cong \KM$ by $\Lambda_i/Z(H)\cong \KM$.  Recall that the $\Lambda_i$ are the subgroups of $H$ defined in it by the same formulas that define the $\Delta_i$ in $G$. Let $A_1$ be the copy of $A$ sitting in $Z(H)$ and $A_2$ the one sitting in $\Lambda_i/Z(H)$. Now
\begin{align*}
Ext(\KM, \KM)&\cong Ext(B(K)\times A_2, B(K)\times A_1))\\
&\cong Ext(B(K),B(K))\times Ext(B(K),A_1)\\
&~~~\times Ext(A_2,B(K))\times Ext(A_2,A_1)\\
&\cong Ext(B(K),A_1)\times Ext(A_2,A_1).
\end{align*}
The last isomorphism holds since $B(K)$ is a divisible group. So to prove that $\Lambda_i$ is a split extension of $Z(H)$ by $\Lambda_i/Z(H)$ we need to prove that $A_1$ splits from the group $\Lambda_i$. 

Let us work inside $G$ and come up with an $L$-sentence which is obviously true in $G$ and states that $\Delta_i(G)$ splits over $A_1$. To that end let
\[\Delta^2_i=\{x^2: x\in \Delta_i\}.\] Thus $\Delta^2_i$ is a definable subgroup of $\Delta_i$. 
 The formula $\phi(x):x^2=1$ defines $A_1\times A_2$ in $\Delta_i$ and already implies the direct decomposition. So
\[\Delta_i=A_1\times A_2 \times \Delta^2_i\]
is a definable direct decomposition of $\Delta_i$. The formulas are uniform with respect to $Th(G)$, so the above implies that $A_1$ as a subgroup of $\Lambda_i$ splits from it.  \qed

\section{Isomorphisms between $T_n$ and abelian deformations of $T_n$}\label{abiso:sec} 
In this section we prove that if $Ext(\RM,\RM)\neq 1$ then  $T_n(R,\bar{f})$ is a genuinely new object for some of the $f_i\in Ext(\RM,\RM)$.

\begin{lemma}\label{split:lem} Assume $E$ is a split abelian extension  of $A$ by $B$, both written additively, defined by the 2-cocycle $f$ and $E'$ is an extension of $A'$ by $B'$ defined by the 2-cocycle $g$. Moreover assume there are isomorphisms $\psi$, $\phi$, and $\eta$ making the following diagram commutative. Then the 2-cocycle $g$ is a 2-coboundary.
\[\begin{CD}
0 @>>> A @>{\mu}>> E @>{\epsilon}>> B @>>> 0 \\
@. @VV{\psi}V @VV{\phi}V  @VV{\eta}V @. \\
0 @>>> A' @>>{\mu'}> E' @>>{\epsilon'}> B' @>>> 0
\end{CD}\]
 \end{lemma}

\begin{proof}
We can easily modify the diagram  in the statement to an equivalence of extensions as follows: 

\[\begin{CD}
0 @>>> A' @>{\mu\circ \psi^{-1}}>> E @>{\eta\circ \epsilon}>> B' @>>> 0 \\
@.  @| @VV{\phi}V  @| @. \\
0 @>>> A' @>>{\mu'}> E' @>>{\epsilon'}> B' @>>> 0
\end{CD}\]

This shows that the cocycle $g$ is cohomologous to the 2-cocycle $g'$ defined by
\[g'(x,y)=\psi^{-1}(f((\eta(x),\eta(y))).\]
By hypothesis $f$ is a 2-coboundary, i.e. there exists a function $h:B\to A$ such that $f(x,y)=h(x+y)-h(x)-h(y)$. since $\eta$ and $\psi^{-1}$ are both group isomorphisms

\[g'(x,y)=\psi^{-1}(h(\eta(x)+\eta(y)))-\psi^{-1}(h(\eta(x)))-\psi^{-1}(h(\eta(y))),\]

proving that $g'$ is a 2-coboundary. But $g$ and $g'$ are cohomologous, which proves that $g$ is 2-coboundary.  
\end{proof} 

\begin{lemma}\label{abdef-notiso:lem} Assume $R$ and $S$ are characteristic zero integral domains with unit. Let $\phi: G=T_n(R, \bar{f}) \to T_n(S)=H$ be an isomorphism of abstract groups. Then $R\cong S$ as rings and all the symmetric 2-cocycles $f_i$ are 2-coboundaries.\end{lemma}
\begin{proof} 
 Since $G'$ and $H'$ are characteristic subgroups of the corresponding groups $\phi$ restricts to an isomorphism of $G'$ onto $H'$. Recall from Lemma~\ref{G':lemma} that
 \[G'=\langle t_{i,i+1}((1-\alpha)\beta), t_{kl}(\beta):1\leq i\leq n-1, 1< k+1<l\leq n, \alpha\in\RM, \beta\in R\rangle,\]
 and 
  \[H'=\langle t_{i,i+1}((1-\alpha)\beta), t_{kl}(\beta):1\leq i\leq n-1, 1< k+1<l\leq n, \alpha\in\SM, \beta\in S\rangle.\]
  In particular all the generators of $UT_n(R)$ (resp. $UT_n(S)$) are in $G'$ (Resp. $H'$) except possibly $t_{i,i+1}(\beta)$ where $2\nmid \beta$, $\beta\in R$ (resp. $\beta\in S$). Consider $h=t_{i,i+1}(\beta)$, $2\nmid \beta$, if such $\beta\in S$ exists. Then $h^2\in H'$. Since $h\in Fitt(H)=UT_n(H)\times Z(H)$ and $Fitt(H)$ is a characteristic subgroup of $H$ and $\phi$ is an isomorphism, there is $g\in Fitt(G)$ such that, $g\in UT_n(R)\times Z(G)$, $g^2\in G'\leq UT_n(R)$ and  $\phi(g)=h$. Since $R$ is a characteristic zero integral domain $UT_n(R)$ is torsion-free. This implies that elements of $UT_n(R)$ have unique roots in $UT_n(R)$ if they have any. Moreover the only element of order 2 in $Fitt(G)$ is $-I_n\in Z(G)\cong \RM$ again since $R$ is an integral domain. So either $g\in UT_n(R)$ or $-I_ng\in UT_n(R)$. This immediately implies that either $h=t_{i,i+1}(\beta)$ or $h=-I_nt_{i,i+1}(\beta)$. Indeed what we discovered is that $\phi$ restricts to an isomorphism of $UT_n(R)\times \ZZ$ onto $UT_n(S)\times \{\pm I_n\}$, where $\phi(- I_n)=- I_n$ and therefore $\phi(UT_n(R))$ is a complement of $\{\pm I_n\}$ in $UT_n(S)\times \{\pm I_n\}$. That is to say there is a derivation $\delta: UT_n(S) \to \{\pm I_n\}$ where $\phi(g)=\phi_1(g)\delta((\phi_1(g))$ for some uniquely determined $\phi_1(g)\in UT_n(S)$. Note that $\delta$ is actually a homomorphism since its range is included in the center $Z(H)$ of $H$. Therefore $\phi_1:UT_n(R)\to UT_n(S)$ is also a homomorphism. It is easy to verify that $\phi_1$ is an isomorphism of $UT_n(R)$ onto $UT_n(S)$. Then by (\cite{beleg99}, Theorem 1.14.1) $R\cong S$. This proves the first part of the claim.
  
  Indeed what we showed above is that there are homomorphisms $\phi_1:UT_n(R)\to UT_n(S)$ and $\delta': UT_n(R)\to \{\pm I_n\}$ where $\phi(g)=\phi_1(g)\delta'(g)$,  $\delta'(g)=\delta(\phi(g))$ for every $g\in UT_n(R)$, and $\delta'(t)=1$ for every $t\in G'$. The later statement implies that 
  \begin{equation}\label{delta-invar:eq} \delta'(x^{-1}gx)=\delta'(g), ~\forall g\in UT_n(R), \forall x\in G.\end{equation}
  
  Next we use the homomorphism $\delta$ above in an obvious way to twist $\phi$ into an isomorphism $\psi:G\to H$ such that $\psi(UT_n(R))=UT_n(S)$. To that end for $x\in D_n(R)$ and $g\in UT_n(R)$ define 
 \[\psi(xg)=\phi(x)\phi(g)\delta'(g).\]   
 Then one can use \eqref{delta-invar:eq} to prove that $\psi$ is a homomorphism which is clearly bijective. Moreover for any $g\in UT_n(R)$  
 \[\psi(g)=\phi_1(g)(\delta'(g))^2=\phi_1(g)\in UT_n(S).\]
 Now it is clear that
 \[D_n(R)\cong \frac{G}{UT_n(R)}\cong \frac{H}{UT_n(S)} \cong D_n(S,\bar{f}).\]
 On the other hand since $\psi(Z(G))=Z(H)$ the isomorphism above takes $Z(G)\cong\RM$ to $Z(H)\cong \SM$. So applying Lemma~\ref{split:lem} and considering that 
 \[Ext((\SM)^{n-1},\SM)\cong \prod_{i=1}^{n-1} Ext(\SM,\SM),\]
 we conclude that all the $f_i$ are 2-coboundaries or equivalently $D_n(S,\bar{f})$ splits over $Z(H)$.      
\end{proof}

\section{Abelian deformations of $T_n(\Q)$ which are not isomorphic to any $T_n(K)$ for any field $K$}\label{Tn(Q):sec}
 
\begin{lemma}\label{Z-inter-Q*:lem}
There exists a countable field $K$, such that $\Q\equiv K$ and also $Ext(K^\times,K^\times)\neq 1$. 
\end{lemma}
\begin{proof} By (\cite{beleg99}, Proposition 2.2.16) there exists a countable ring $R$ where $R\equiv \Z$ and $Ext(R^+,R^+)\neq 0$. That is because for such a ring $R$, $R^+=A\oplus D$ where $A\neq 0$ is a reduced abelian group and $D\neq 0$ is a countable torsion-free divisible abelian group. Then by (\cite{fox}, 54.5) $Ext(D,A)\neq 0$, which implies that $Ext(R^+,R^+)\neq 0$.

\noindent{\bf Claim:} Let $K$ be the field of fraction of the integral domain $R$ discussed above. Then $K^\times \cong \{\pm 1\} \times A \times D$ where $A$ and $D$ are both non-trivial, $A$ is reduced and $D$ is countable, torsion-free and divisible.

\noindent{\it Proof of the claim.} It is well-known that $\QM= \{\pm 1\}\times C$ where $C\cong\prod_{i\in \omega}C_{\infty}$, and $C_{\infty}$ is the infinite cyclic group. Firstly we show that $C$ is definable in $\QM$ uniformly with respect to $Th(\Q)$. By a theorem of \cite{julia} the ring $\Z$ in definable in $\Q$.  The arithmetic $\N=\langle |\N|, +,\times , 0,1\rangle $ is definable in $\Z$. In particular the natural order on $\Z$ is definable in $\Z$. Since $\Z$ is definable in $\Q$ it is easy to see that natural order on $\Q$ is definable in $\Q$. Now $\Q^\times$ is definable in $\Q$ as the set of invertible elements. The subgroup $C$ is definable in $\QM$ as the set of all positive elements of $\Q$. The subgroup $\{\pm 1 \}$ is obviously definable in $\QM$.  

 We note that $\Q$ is interpretable in $\Z$ and the same formulas that interpret $\Q$ in $\Z$ interpret $K$ in $R$. So indeed $\Q\equiv K$. The formulas that define $\Z$ in $\Q$ define a subring $R$ of $K$ in $K$ which is elementarily equivalent to $\Z$. Again the formulas that define $C$ in $\Q$ define a subgroup $E$ of the multiplicative group $\KM$ of $K$ in $K$. So clearly $\KM=\{\pm 1\}\times E$. Next we will show that $E$ has a direct factor which is isomorphic as an abelian group to $R^+$. 
 By a well-known result of G\"odel all primitive recursive functions on $\N$ are arithmetically definable, in particular the predicate $z=2^x$ is arithmetically definable. That is, there is a first-order formula $exp(z,y,x)$ of the language of arithmetic so that 
\[z=2^x \Leftrightarrow \N\models exp(z,1+1,x).\]
In particular it can be seen that the set $2^\N=\{2^n: n\in \N\}$ has a definable arithmetic structure also denoted by $2^\N$. Since $\N$ is definable in $\Z$ so is $2^\N$. This implies that $2^\Z$ is definable in $\Q$ and also it carries a ring structure interpretable in $\Q$ and isomorphic to the ring of integers $\Z$. on the other hand $2^\Z \subset \Q^{\times}$. The formulas that define an arithmetic structure on $2^\Z$ define in $K$ a  subgroup $2^R\leq E < \KM$ of $K$ with a definable ring structure isomorphic to $R$. Now $R^+\cong 2^R \cong A'\times D'$  where both $A'$ and $D'$ are non-empty and $D'$ is torsion-free divisible. This means that $C= A \times D$ where both $A$ and $D$ are non-empty, $A$ is reduced torsion-free and $D$ is countable, torsion-free and divisible. This ends the proof of the claim.

As a corollary of the claim and the fact $Ext(D,A)\neq 0$ we can conclude that $Ext(\KM,\KM)\neq 1$.        \end{proof}

\emph{Proof of Theorem~\ref{elemnotiso-Q:thm}.} Pick the field $K$ found in Lemma~\ref{Z-inter-Q*:lem}. We can clearly find a CoT 2-cocycle $f\in S^2(\KM,\KM)$ which is not a 2-coboundary. Set, say, $f_1=f$ and the let the rest of the $f_i$ be trivial and form $H=T_n(K,\bar{f})$. By Theorem~\ref{main(improved):thm}, $G\equiv H$. But by the construction of $K$ and Lemma~\ref{abdef-notiso:lem}, $G\ncong H$, otherwise $f$ would be a 2-coboundary.\qed 
\section{ The case of $T_n(\OR)$, $\OR$ ring of integers of a number field}\label{Tn(O):sec}

In subsection~\ref{case1:subsec} we will prove part (1) of Theorem~\ref{charthm-ringofinter:thm} which is actually easy given what we have developed so far. In
subsection~\ref{case2:subsec} we will address the second part of the theorem.
\subsection{The case of finite $\OM$}\label{case1:subsec}
\emph{Proof of Theorem~\ref{charthm-ringofinter:thm}, Part 1.} By Theorem~\ref{main(improved):thm}, $H\cong(T_n(S,\bar{f}))\cong E_n\ltimes UT_n(S)$, where $E_n\equiv D_n$. But here $D_n$ is finite, so $E_n\cong D_n$ meaning all $f_i$ are coboundaries. So $T_n(S,\bar{f})\cong T_n(S)$. \qed

\subsection{The case of infinite $\OM$}\label{case2:subsec} 
\begin{lemma}\label{x^Z:inter:lem} If $\OM$ is infinite then for some $\lambda\in \OM$ of infinite order The set $\lambda^\Z$ of all integer powers of $\lambda$ is definable in $\OM$. Indeed there is a ring isomorphism $\lambda^\Z\cong \Z$ which is interpretable in $\OR$.\end{lemma}
The proof is actually a simple modification of the interpretation of the arithmetic in a metabelian polycyclic group satisfying specific conditions in~\cite{Romanovskii}. We need to quote a few lemmas from~\cite{Romanovskii} before we get to the proof. Let $K$ be the field of fractions of $\OR$ and let $deg(K/\Q)=s$.
\begin{lemma}[Lemma 2,~\cite{Romanovskii}] \label{RomLem2:lem} Let $\lambda_1, \ldots, \lambda_s$ be distinct elements of $\OR$, $\delta\in \OR$ where $\delta\neq 0$. Then there are only finitely many $\alpha\in \OR$ for which the elements $\alpha-\lambda_1$, \ldots $\alpha-\lambda_s$ divide $\delta$ in $\OR$. \end{lemma}
\begin{lemma}[Lemma 3,~\cite{Romanovskii}] \label{RomLem3:lem} Let $X$ be a finite set of elements of $\OR$ other than 0 or 1. Let $\lambda$ be a nonzero element of $\OR$ that is not a root of unity. Then there is a natural number $n$ such that $\lambda^n-\alpha\notin \OM$ for all $\alpha\in X$.\end{lemma}
We say that two elements $\alpha$ and $\beta$ of $\OR$ are \emph{associated} with each other if there is a unit $\gamma\in \OR$ such that $\alpha=\gamma\beta$.
\begin{lemma}[Lemma 4,~\cite{Romanovskii}]\label{RomLem4:lem} There is a natural number $r$ such that if $\epsilon$ is a nonzero element of $\OR$ and not a root of unity and if $\lambda$ is sufficiently large power of it, then for any natural number $m$ and $n$ the fact that $\lambda^m - 1$ and $\lambda^n-1$, $\lambda^{2m} - 1$ and $\lambda^{2n}-1$, \ldots, $\lambda^{rm} - 1$ and $\lambda^{rn}-1$ are pairwise associated in $\OR$ implies $m=n$.\end{lemma}
\begin{lemma}\label{B-def:lemma} There is a non-trivial torsion-free definable subgroup $B$ of $\OM$ such that every non-trivial definable subgroup of $B$ is of finite index in $B$.\end{lemma}
\begin{proof}
By Dirichlet's units theorem $\OM$ is a finitely generated abelian group. So there is natural number $k$ for which $B_1=(\OM)^k$ is torsion-free. Either $B_1$ satisfies the asked condition or there is a non-trivial definable subgroup $B_2$ of infinite index in $B_1$. Since the rank of $B_2$ as a free abelian group has to be strictly less than that of $B_1$ the process has stop at some stage. Then we obtain the desired subgroup $B$.
\end{proof} 
 
\emph{Proof of Lemma~\ref{x^Z:inter:lem}.}  Note that since $\OM$ is definable in $\OR$, the metabelian polycyclic group $G_1= \OR\rtimes\OM$ is interpretable in $\OR$ where the action is just the right multiplication. We work with a definable subgroup $G_2=\OR\rtimes B$ where $B\leq \OM$ is the subgroup obtained in Lemma~\ref{B-def:lemma}. By $\alpha|\beta$ we mean $\beta$ is divisible by $\alpha$ in $\OR$. Here we have the benefit of having the whole expressive power of the first-order theory of rings at our disposal, which makes this proof simpler than the corresponding one in \cite{Romanovskii}. So for a $\lambda$ as in Lemma~\ref{RomLem4:lem} define the predicate $\Psi$ by
\begin{align*}
\Psi(\alpha,\beta,\delta,a) &\Leftrightarrow \alpha,\beta,\delta \in  B \wedge a\in \OR \wedge \alpha,\delta\neq 1\\
&\wedge \alpha\lambda-1, \ldots, \alpha\lambda^s-1|\delta-1 \wedge 1+(\beta-1)\alpha|a.\end{align*}
Clearly this can be transformed to a first-order formula of the language of rings. Let us set 
\[M=\{\lambda^n: n\in \N\}.\] Now we claim that the equivalence
\[\gamma\in M \Leftrightarrow \exists \beta, \delta, a \{\Psi(\lambda,\beta,\delta,a)\wedge \forall \alpha[\Psi(\alpha,\beta,\delta,a)\to \alpha=\gamma\vee \Psi(\alpha\lambda,\beta,\delta,a)]\},\]
holds. In particular $M$ is a definable subset of $G$. So suppose, to get a contradiction, that for certain $\beta$, $\delta$ and $a$ the right hand side of the equivalence is valid but $\gamma\notin M$. Therefore the inductive form of the formula on the right hand side of the equivalence implies that for every $\alpha\in M$ the formula $\Psi(\alpha,\beta,\delta,a)$ holds. Hence for infinitely many distinct $\alpha$, $\alpha\lambda-1, \ldots \alpha\lambda^s-1|\delta-1$, contradicting Lemma~\ref{RomLem2:lem}. So indeed $\gamma\in M$.

Conversely, let $\gamma= \lambda^m$, where $m$ is a natural number. We show that the right-hand side of the equivalence holds. We set $\delta=\lambda^{(m+s)!}$. Let $X$ be the collection of all elements $\alpha\neq 1$ in $B$ such that $\alpha\lambda-1, \ldots, \alpha\lambda^s-1$ divide $\delta-1$. The set $X$ contains $\lambda, \ldots ,\lambda^m$. By Lemma~\ref{RomLem2:lem} $X$ is a finite set. Let $\alpha, \xi \in X$, $\alpha\neq\xi$. Now $B_{\alpha,\xi}=\{\mu\in B:\alpha-\xi|\mu-1\}$ is a non-trivial definable subgroup of $B$. Definability is clear. Obviously if $\mu \in B_{\alpha,\xi}$ then $\alpha-\xi|\mu^{-1}(\mu-1)$, which implies that $\mu^{-1}\in B_{\alpha,\xi}$. To prove that $B_{\alpha,\xi}$ is closed under product one may use the identity $(\mu_1-1)(\mu_2-1)=\mu_1\mu_2-1-(\mu_1-1)-(\mu_2-1)$. By Lemma~\ref{B-def:lemma} $B_{\alpha,\xi}$ is a subgroup of finite index in $B$. Since $X$ is finite the intersection of all $B_{\alpha,\xi}$,$\alpha,\xi\in X$, $\alpha\neq \xi $ is a subgroup of finite index in $B$. Let $\beta_1$ be a nontrivial element of this intersection. By Lemma~\ref{RomLem3:lem} there is a natural number $n$ such that the elements $\beta_1+\alpha^{-1}-1$ $(\alpha \in X)$ are not invertible in $\OR$. Note that $1+(\beta_1^n-1)\alpha=\alpha(\beta_1^n+\alpha^{-1}-1)$. Let $\beta=\beta_1^n$ and
\[a=[1+(\beta-1)\lambda]\cdots [1+(\beta-1)\lambda^m].\]
Let $\alpha\in X$, $\alpha\neq \lambda, \ldots , \lambda^m$. Then for every $i=1, \ldots, m$ the algebraic integers $1+(\beta-1)\lambda^i$ and $1+(\beta-1)\alpha$ are coprime in $\OR$. For if a prime divisor $d$ divides these numbers, then it divides their difference $(\beta-1)(\lambda^i-\alpha)$. Since $\lambda^i-\alpha$ divides $\beta-1$, $d$ divides $\beta-1$. From the condition that $d$ divides $1+(\beta-1)\alpha$ it follows that $d$ divides $1$, which contradicts the fact that $d$ is a prime. This show that $1+(\beta-1)\alpha$ is not invertible in $\OR$ and is prime to the elements $1+(\beta-1)\lambda^i$, $i=1,\ldots,m$. So this element does not divide $a$ and consequently $\Psi(\alpha,\beta,\delta,a)$ does not hold. On the other hand $\Psi(\lambda^i,\beta,\delta,a)$ holds for $i=1, \ldots m$. It is clear now that the right hand side of the equivalence holds for the chosen $\beta$, $\delta$ and $a$. This concludes the definability of $M$. Now one copies the rest of the paragraph 2.3 on Page 131 from~\cite{Romanovskii} to define  an arithmetic structure on $M=\lambda^\N$. Next we can define $M^\times=\lambda^\Z$ in $\OM$ and show that it has a definable structure isomorphic to the ring of integers $\Z$. This finishes the proof of the lemma.\qed

\begin{proposition} \label{extneq0:prop}Assume $\OR$ is the ring of integers of a number field. Then there is a countable non-standard model $R$ of $Th(\OR)$ such that $Ext(\RM,\RM)\neq 1$.
\end{proposition}
\begin{proof}Again by Dirichlet units theorem $\OM$ is a finitely generated abelian group. So the maximal torsion subgroup $T(\OM)$ of is definable in $\OM$ as $\{x\in \OM: x^k=1\}$, $k$ the exponent of $B(\OM)$. Let $R$ be any countable non-standard model of $\OR$. Then $T(\RM)\cong T(\OM)$ where $\RM/T(\RM)$ is torsion-free. So by Lemma~\ref{tbytf:lem}, $\RM\cong T(\RM) \times B(\RM)$. Recall that there exists a formula $\Phi(\lambda,x)$ of $L$ that defines $M(\lambda)=\lambda^\Z$ and a ring structure on it in $\OM$. We note that $\lambda$ has to satisfy conditions of Lemma~\ref{RomLem4:lem}. Indeed we only need just pick a sufficiently large power, say a fixed $n$, of any non-trivial element of $B(\OM)$. As mentioned above $B(\OM)$ is interpretable in $\OR$ and the same formulas interpret $B(\RM)$ in $R$. So the formula $\exists \lambda (\lambda\in B(\OM) \wedge \Phi(\lambda,x))$ will interpret a subgroup $M^*$ of $\RM$ in $R$ which also has a ring structure.  As such $M^*$ is a countable non-standard model of $\Z$. The rest of the argument is similar to the final part of the proof of Lemma~\ref{Z-inter-Q*:lem}.     
\end{proof}

\emph{Proof of Theorem~\ref{charthm-ringofinter:thm}, Part 2.} The Proof is similar to the proof of Theorem~\ref{elemnotiso-Q:thm} but now combining and Lemma~\ref{abdef-notiso:lem} and Proposition~\ref{extneq0:prop}.  \qed

\end{document}